\documentclass[11pt]{article}%
\usepackage{amsmath}
\usepackage{amssymb}
\usepackage{amsfonts}
\usepackage{graphicx}%
\providecommand{\U}[1]{\protect\rule{.1in}{.1in}}
\newtheorem{theorem}{Theorem}[section]

\newtheorem{corollary}[theorem]{Corollary}

\newtheorem{definition}[theorem]{Definition}

\newtheorem{lemma}[theorem]{Lemma}

\newtheorem{remark}[theorem]{Remark}

\newenvironment{proof}[1][Proof]{\textbf{#1.} }{\ \rule{0.5em}{0.5em}}
\newdimen\dummy
\dummy=\oddsidemargin
\ifx\pdfoutput\relax\let\pdfoutput=\undefined\fi
\newcount\msipdfoutput
\ifx\pdfoutput\undefined\else
\ifcase\pdfoutput\else
\ifx\paperwidth\undefined\else
\ifdim\paperheight=0pt\relax\else\pdfpageheight\paperheight\fi
\ifdim\paperwidth=0pt\relax\else\pdfpagewidth\paperwidth\fi
\fi\fi\fi
\begin{document}

\title{Coverings of graded pointed Hopf algebras}
\author{Esther Beneish and William Chin\\University of Wisconsin-Parkside, Kenosha Wisconsin \\DePaul University, Chicago, Illinois}
\maketitle

\begin{abstract}
We introduce the concept of a covering of a graded pointed Hopf algebra. The
theory developed shows that coverings of a bosonized Nichols algebra can be
concretely expressed by biproducts using a quotient of the universal coalgebra
covering group of the Nichols algebra. If there are enough quadratic
relations, the universal coalgebra covering is given by the bosonization by
the enveloping group of the underlying rack.

\end{abstract}

\section{Introduction}

Nichols algebras play a crucial role in the classification results for pointed
Hopf algebras over abelian groups, which vastly generalize the theory of
quantized enveloping algebras. They appear in several recent investigations of
pointed Hopf algebras over nonabelian groups. Nichols algebras are seen to be
fundamental objects, appearing in the study of the cohomology of flag
manifolds. See e.g., \cite{AG2003a},\cite{AS2002},\cite{GHV2011}%
,\cite{HLV2012},\cite{FK99} and references therein.

A Nichols algebra $B(V)$ is constructed from a braided vector space $V,$ where
$V$ in turn is, in the most studied cases, a Yetter-Drinfeld module over some
group $G$. The Nichols algebra depends only on the braiding $c:V\otimes
V\rightarrow V\otimes V$ and there may be many groups yielding the braiding.
In this paper, we study groups that can arise given a fixed Yetter-Drinfeld
module $V\in%
\genfrac{}{}{0pt}{}{G}{G}%
\mathcal{YD}$.

We assume that $V$ is a link-indecomposable Yetter-Drinfeld module over a
group $G$, and that $X$ is the corresponding rack with cocycle $q:X\times
X\rightarrow\mathbb{\Bbbk}^{\times}$. $V\ $also takes the form $\oplus
_{i}M(g_{i},\rho_{i})\cong(\mathbb{\Bbbk}X,c^{q})$ as braided vector spaces.
That is, we are assuming that $V$ is of \textit{rack type} with braiding
$c=c^{q}:V\otimes V\rightarrow V\otimes V.$ Here $g_{i}\in G$ and $\rho_{i}$
is a one-dimensional representation of the centralizer of $g_{i}$ in $G.$ The
equivalence of the two constructions of $V$ is explained by \cite[Theorem
4.14]{AG2003a} (in greater generality)$.$

The theory of coalgebra coverings in \cite{Chin2010} yields indecomposable
coalgebra coverings of $B(V)$. Such coverings take the form $B(V)\rtimes\Bbbk
G\rightarrow B(V)$ where $G$ is a homomorphic image of the universal coalgebra
covering group $\tilde{G}$ for $B\left(  V\right)  $ and $\rtimes$ is the
smash coproduct (or "co-smash product") coalgebra. Let $G_{X}$ denote the
enveloping group of the rack $X$ (see \cite{AG2003a},\cite{GHV2011}). We have
in general for a braiding of rack type that $V\in%
\genfrac{}{}{0pt}{}{\Bbbk G_{X}}{\Bbbk G_{X}}%
\mathcal{YD}$ and therefore there is a surjection $\tilde{G}\rightarrow G_{X}$
and a corresponding coalgebra surjection $B(V)\rtimes\Bbbk\tilde{G}\rightarrow
B(V)\rtimes\Bbbk G_{X}.$ We show in Corollary \ref{uni} that a certain
quotient $B(V)\#\hat{G}$ of $B(V)\#\tilde{G}$ serves as the universal Hopf
covering of $B(V)\#G$. Thus we have a \textit{Hopf algebra covering} of the
graded Hopf algebra the sense exhibited in \cite{Lentner2013} arising from a
central extension
\[
1\rightarrow Z\rightarrow\hat{G}\rightarrow G\rightarrow1.
\]
It is known by \cite[Lemma 3.4]{AFGV2011} that $G_{X}$ also is universal for
$V$ so we have in fact that $\hat{G}=G_{X}.$

The corresponding rack braiding $c:X\times X\rightarrow X\times X$ decomposes
into $c$-orbits $\mathcal{O=O}_{x,y}$. Recall that with $c=c^{q}:V\otimes
V\rightarrow V\otimes V,$ $\ker(1+c)$ defines the quadratic term $B(V)(2)$.
When $\ker(1+c^{{}})\cap V_{\mathcal{O}}\mathcal{\neq}0$ for each nontrivial
$c$-orbit $\mathcal{O}$, we say that $B(V)$ has \textit{a full set of
quadratic relations. (}We are ignoring the orbits of the form $\{(x,x)\}.)$ We
show in Theorem \ref{full} that, for a Nichols algebra with a full set of
quadratic relations, $G_{X}=\tilde{G}.$ Thus $B(V)\#G_{X}\rightarrow B(V)$ is
the universal coalgebra covering, which factors through the Hopf algebra
surjection $B(V)\#\Bbbk G_{X}\rightarrow B(V)\#\Bbbk G$. This means that every
group $G$ such that $B(V)\rtimes\Bbbk G$ is an indecomposable coalgebra arises
as a homomorphic image of $G_{X}.$ We have thus a universal Hopf covering for
the Nichols algebra $B(V).$ The use of $G_{X}$ as the "principal" grading
group is mentioned in \cite[Lemma 3.4]{AFGV2011} and suggested in
\cite{GaV2014}.

In case $G_{X}=\tilde{G},$ we remark that the relations in the definition of
$\tilde{G}$ of degree greater than two are superfluous. In all known examples
with finite dimensional $B(V)$ we find that $G_{X}=\tilde{G},$ and we
conjecture that this is always the case.

To compare with \cite{GHV2011}, recall that $B(V)$ is said to have
\textit{many quadratic relations} if the $\dim\ker(1+c)\geq d(d-1)/2$ where
$d=\dim V.$ The two notions pertaining to quadratic relations do not appear to
be comparable, see for example \ref{FK Ex}.

We are interested in finite racks $X$ that are unions of conjugacy classes of
the group $G$ and we assume that $X$ generates $G$. In this latter case the
Yetter-Drinfeld module is said to be \textit{link-indecomposable}, so the
covering Hopf algebras will also be link-indecomposable (i.e. indecomposable
as coalgebra). We provide families of link-indecomposable Hopf algebras
corresponding to the covering Hopf algebras of finite-dimensional Nichols
algebras. These Hopf algebras are finite-dimensional when the covering group
is finite. This addresses the question posed in \cite{Montgomery95} by
including all finite dimensional covering Hopf algebras of known examples of
link indecomposable bosonized Nichols algebras. Specifically, the families of
finite-dimensional Hopf algebras produced arise from the finite homomorphic
images $G_{X}$ by finite index subgroups of its center $Z(G_{X}).$ Some of
these examples are given in \cite[Section 6]{AG2003a}.

Our main references for racks, Nichols algebras, and pointed Hopf algebras are
\cite{AG2003a},\cite{AS2002}; also see \cite{GHV2011}, \cite{HLV2012}.

It would be interesting to try to remove the condition that $\dim\rho=1$ from
our results. Another direction would be to examine more general situations in
which $\tilde{G}=G_{X}.$ It might also be of interest to extend results here
to liftings to nongraded pointed Hopf algebras.

\section{Path coalgebras and pointed coalgebras}

We refer to \cite{Chin2004},\cite{Chin2010} for basics of pointed coalgebras
and path coalgebras. The \textbf{path coalgebra} $\Bbbk Q$ of a quiver $Q$ is
defined to be the span of all paths in $Q$ with coalgebra structure
\[%
\begin{array}
[c]{c}%
\Delta(p)=\sum_{p=p_{2}p_{1}}p_{2}\otimes p_{1}+t(p)\otimes p+p\otimes s(p)\\
\varepsilon(p)=\delta_{|p|,0}%
\end{array}
\]
where $p_{2}p_{1}$ is the concatenation $a_{t}a_{t-1}...a_{s+1}a_{s}%
...a_{1}\,$of the subpaths $p_{2}=a_{t}a_{t-1}...a_{s+1}$ and $p_{1}%
=a_{s}...a_{1}$ ($a_{i}\in Q_{0}$)$.$ Here $|p|=t$ denotes the length of $p$
and the starting vertex of $a_{i+1}\,$ is required to be the end of $a_{i}.$
Thus the vertices $Q_{0}$ are group-like elements, and if $a$ is an arrow
$g\leftarrow h$, with $g,h\in Q_{0},$ then $a$ is a $(g,h)$- skew primitive,
i.e., $\Delta a=g\otimes a+a\otimes h.$ It follows that $\Bbbk Q$ is pointed
with coradical $(kQ)_{0}=kQ_{0}\,$and the degree one term of the coradical
filtration is $(\Bbbk Q)_{1}=\Bbbk Q_{0}\oplus\Bbbk Q_{1}.$ Moreover the
coradical grading $\Bbbk Q=\bigoplus\limits_{n\geq0}\Bbbk Q_{n}$ is given by
path length. The path coalgebra may be identified with the cotensor coalgebra
$C(\Bbbk Q_{1})=\oplus_{n\geq0}(\Bbbk Q_{1})^{\Box n}$ of the $\Bbbk Q_{0}%
$-bicomodule $\Bbbk Q_{1},$ cf. \cite{Nichols78}.

Define the (\textit{Gabriel- }or\textit{\ Ext-}) \textit{quiver} of a pointed
coalgebra $B$ to be the directed graph $Q=Q_{B}$ with vertices $Q_{0} $
corresponding to group-likes and $\mathrm{\dim}_{\Bbbk}P_{h,g}$ arrows from
$h$ to $g$, for all $h,g\in Q_{0}$. We will view $B$ as a subcoalgebra of the
path coalgebra of its Gabriel quiver, with the same degree one coradical term.
If $B$ has a unique group-like element with a space of primitives of dimension
$n$, then the quiver is the $n$-loop quiver. The path coalgebra is the cofree
pointed irreducible coalgebra (and the path algebra is the free algebra on $n$ generators).

The indecomposable components ("blocks") of $B$ are coalgebras that are the
direct sums of injective indecomposables having socles in a given graph
component. Therefore $B$ is indecomposable as a coalgebra if and only if it is
"link-indecomposable", if and only if its quiver is connected as a graph. In
\cite{Montgomery95} it is shown that a pointed coalgebra is a crossed product
over the principal block subcoalgebra, i.e. the one containing $1_{G(B)}.$

\section{Coverings}

We first summarize results from \cite{Chin2010} concerning coverings of
pointed coalgebras. An analogous version for bound quivers finite-dimensional
algebras can be found in \cite{Green83} and \cite{Martinez83}.

Let $B\subseteq\Bbbk Q$ be a subspace. Let $b=\sum_{i\in I}\lambda_{i}p_{i}\in
B(x,y)$ with $x,y\in Q_{0}$ and distinct paths $p_{i}.$ We say that $b$ is a
\textit{minimal element\ }of $B$ if $\sum_{i\in I^{\prime}}\lambda_{i}%
p_{i}\notin B(x,y)$ for every nonempty proper subset $I^{\prime}\subset I$ and
$|I|\geq2.$ Clearly every element of $B$ is a linear combination of paths and
minimal elements. Let $\min(B)$ denote the set of minimal elements of $B$.

For an admissible ideal $I$ of a path algebra, let $A=\Bbbk^{Q}/I$ denote an
admissible quotient of the path algebra $\Bbbk^{Q}$ with ideal of relations
$I$. Let $B\subseteq\Bbbk Q$ be an admissible subcoalgebra of the path
coalgebra $\Bbbk Q$.

Fix a base vertex $x_{0}\in Q_{0}$. We define a symmetric relation $\sim$ on
paths by declaring $p\sim q$ if there is a minimal element $b=\sum_{i\in
I}\lambda_{i}p_{i}\in B(x,y)$ where the $p_{i}$ are distinct paths,
$\lambda_{i}\in\Bbbk$, $x,y\in Q_{0}$ and $p=p_{1}$, $q=p_{2}.$ We define
$N(B,x_{0})$ to be the subgroup of $\pi_{1}(B,x_{0})$ generated by equivalence
(homotopy) classes of walks $w^{-1}p^{-1}qw$ where $p,q$ are paths in $Q(x,y)$
with $p\sim q$ and $w$ is a walk from $x_{0}$ to $x.$

It is easy to see that $N(B,x_{0})$ is a normal subgroup of $\pi_{1}%
(B,x_{0}).$ Explicitly, if $w^{-1}p^{-1}qw$ is closed walk as above and
$[v]\in\pi_{1}(B,x_{0})$ where $v$ is a closed walk at $x_{0}$, then $wv$ is a
path from $x_{0}$ to $x$ and $[v^{-1}w^{-1}p^{-1}qwv]=[(wv)^{-1}p^{-1}q\left(
wv\right)  ]\in N(B,x_{0}).$

Consider a Galois covering $F:\tilde{Q}\rightarrow Q$ of quivers with
automorphism group $G$ and lifting $L$. Let $\tilde{B}$ denote the $\Bbbk
$-span of $\{L(b)|L$ a lifting, $b\in B$ a minimal element or a path\}. We say
that the restriction $F:\tilde{B}\rightarrow B$ is a \textit{Galois coalgebra
covering} if every minimal element of $B$ can be lifted to $\tilde{B}$ in the
following sense: for every minimal element \ $b\in B(x,y)$ with $x,y\in Q_{0}$
and $\tilde{x}\in F^{-1}(x),$ there exists $\tilde{y}\in\tilde{Q}_{0}$ and a
minimal element $\tilde{b}\in\tilde{B}(\tilde{x},\tilde{y})$ such that
$F(\tilde{b})=b.$ All coverings in this paper are assumed to be Galois.

Let $B\subseteq\Bbbk Q$ be a pointed coalgebra. Then there exists a Galois
coalgebra covering $F:\tilde{B}\rightarrow B$ $,$ the \textit{universal
coalgebra covering} of $B\subseteq\Bbbk Q,$ such that for every Galois
coalgebra covering $F^{\prime}:B^{\prime}\rightarrow B$, there exists a Galois
coalgebra covering $E:\tilde{B}\rightarrow B^{\prime}$ such that the following
diagram commutes.%

\[%
\begin{array}
[c]{ccccc}%
\tilde{B} &  & \overset{E}{\xrightarrow{\hspace*{.75cm}}} &  & B^{\prime}\\
& \underset{F}{\searrow} &  & \underset{F^{\prime}}{\swarrow} & \\
&  & B &  &
\end{array}
\]

The fundamental example of a covering is given as follows. Let $B\subseteq
\Bbbk Q$ be a homogenous admissible subcoalgebra with respect to the grading
given by an arrow weighting $\delta:Q_{1}\rightarrow G$. If $b=\sum_{i\in
I}\lambda_{i}p_{i}\in B(x,y)$ is a minimal element, then it is necessarily
homogeneous in the $G$-grading. Consider the canonical map $F:\Bbbk
Q\rtimes\Bbbk G\rightarrow\Bbbk Q$ defined by $F(p\rtimes g)=p$ and consider
the restriction to $B\rtimes\Bbbk G\rightarrow B.$ Then under the
identification of $\Bbbk Q\rtimes\Bbbk G$ with $\Bbbk\tilde{Q}$ we easily see
that $\tilde{B}=B\rtimes\Bbbk G.$ The liftings of minimal element $b\in B$ are
given by $b\rtimes g$ with $g\in G.$

\begin{theorem}
[\cite{Chin2010}]\label{Cov}The following are equivalent for a subcoalgebra
$B\subseteq\Bbbk Q $ and Galois quiver covering $F:\tilde{Q}\rightarrow
Q.$\newline(a) $B$ is a homogeneous subcoalgebra of $\Bbbk Q.$\newline(b)
$N(B,x_{0})\subseteq F_{\ast}(\pi_{1}(\tilde{Q},\tilde{x}_{0}))$ for all
$x_{0}\in Q,$ $\tilde{x}_{0}\in\tilde{Q}$ with $F(\tilde{x}_{0})=x_{0}%
.$\newline(c) $F:\tilde{B}\rightarrow B$ is a Galois coalgebra
covering.\newline(d) $B$ is a homogenous subcoalgebra of $\Bbbk Q$ and the
grading is connected.
\end{theorem}

\begin{theorem}
[\cite{Chin2010}]The universal covering of the coalgebra $B\subseteq\Bbbk Q$
is isomorphic to $B\rtimes\Bbbk\tilde{G}\rightarrow B$ where $\tilde{G}%
=\frac{\pi_{1}(\tilde{Q},\tilde{x}_{0})}{N(B,x_{0})},$ $(\tilde{Q},\tilde
{x}_{0})$ is the universal covering quiver of $(Q,x_{0})$ (with base vertices
$\tilde{x}_{0}$ and $x_{0}$), and $B\rtimes\Bbbk\tilde{G}$ is indecomposable
as a coalgebra if $B$ is.
\end{theorem}

\subsection{Hopf coverings}

The idea of Hopf covering comes from the following simple observation.

\begin{lemma}
Let $B(V)$ be a Nichols algebra with a link-indecomposable Yetter-Drinfeld
module $V\in%
\genfrac{}{}{0pt}{}{\Bbbk G}{\Bbbk G}%
\mathcal{YD}.$ Let $Z$ be a normal subgroup of $G$ acting trivially on $V$.
Then $Z$ is central in $G$. Let $\bar{V}$ be a copy of $V$ graded by $\bar
{G}=G/Z$ via $\bar{V}_{gZ}=\oplus_{t\in gZ}V_{t}$ for $gZ\in\bar{G},$ and make
$\bar{V}$ a $\bar{G}$-module by factoring by $Z$. Then $\bar{V}\in%
\genfrac{}{}{0pt}{}{\bar{G}}{\bar{G}}%
\mathcal{YD}$ and $V$ is isomorphic to $\bar{V}$ as a braided vector space.
Therefore $B(V)$ and $B(\bar{V})$ are isomorphic as braided Hopf algebras and
there is a surjection $B(V)\#\Bbbk G\rightarrow B(\bar{V})\#\Bbbk\bar{G}$ on
bosonized Nichols algebras given by $b\#g\mapsto b\#gZ,$ for all $b\in B(V),$
$g\in G.$
\end{lemma}

\begin{proof}
We just point out that $Z$ being central is equivalent to the condition
$z.V_{g}=V_{zgz^{-1}}=V_{g}$ for all $g\in G,$ $z\in Z.$

The conclusions of the lemma hold even in the case that $V$ is not
finite-dimensional. Also if $V$ is not assumed to be link-indecomposable in
the lemma above, then we may replace "$Z$\text{ is central in }$G"$ by "$Z$ is
in the centralizer of $\{g\in G|V_{g}\neq0\}".$
\end{proof}

\begin{definition}
When $\tilde{A}=\oplus_{n\geq0}\tilde{A}_{n}=R\#\Bbbk\tilde{G}$ and
$A=\oplus_{n\geq0}A_{n}=R\#\Bbbk G$ are coradically graded pointed Hopf
algebras, which are bosonizations of the braided graded Hopf algebra $R$ as in
\cite{AS2002}, and $f:\tilde{G}\rightarrow G$ is a group surjection, we say
that the Hopf algebra map $F:\tilde{A}\rightarrow A$ is a \textit{Hopf
covering} if $F(r\#\tilde{g})=r\#f(g)$ for all $r\in R$ and $\tilde{g}.$ We
say that $\tilde{A}$ is a \textit{covering Hopf algebra} of $A,$ with covering
group $\tilde{G}.$ A universal Hopf covering is one that is universal among
Hopf coverings of $A$.
\end{definition}

The following result specifies the Hopf coverings of a bosonization of a
Nichols algebra.

\begin{theorem}
Let $B(V)\#\Bbbk G$ be the bosonization for a link-indecomposable $V\in%
\genfrac{}{}{0pt}{}{\Bbbk G}{\Bbbk G}%
\mathcal{YD}.$ Let $\tilde{G}$ the universal coalgebra covering group of
$B(V)$ and write $G=\tilde{G}/N$ where $N\vartriangleleft\tilde{G}$. The
covering Hopf algebras of $B(V)\#\Bbbk G$ are of the form $B(V)\#\Bbbk
\tilde{G}/M$ where $M\vartriangleleft\tilde{G}$ and $[N,\tilde{G}]\subseteq
M\subseteq N.$
\end{theorem}

\begin{proof}
We have the universal coalgebra covering
\[
B(V)\rtimes\Bbbk\tilde{G}\rightarrow B(V)\rtimes\Bbbk G\rightarrow B(V)
\]
by \cite{Chin2010}, see Theorem \ref{Cov} and \cite{Chin2010}. So we may write
$G=\tilde{G}/N$, $N\vartriangleleft\tilde{G}.$ Now consider the set of
homomorphic images $H=\tilde{G}/M$ of $\tilde{G}$ where $M\in\mathcal{C}$, and
$\mathcal{C}$ is defined as
\begin{align*}
\mathcal{C}  &  \mathcal{=}\mathcal{\{}M\vartriangleleft\tilde{G}|N\supseteq
M\text{ and }N/M\text{ central in }\tilde{G}/M\}\\
&  =\mathcal{\{}M\vartriangleleft\tilde{G}|[N,\tilde{G}]\subseteq M\}
\end{align*}
using the group theoretic commutator. It is evident that the unique minimal
element of $\mathcal{C}$ is just $[N,\tilde{G}].$ Now for such a group $H,$ it
follows from the Lemma that $B(V)\#\Bbbk H$ is a bosonization where the action
is lifted from the action of $G,$ and the grading is inherited from the
$\tilde{G}$ grading. Thus we obtain the Hopf coverings $B(V)\#\Bbbk\tilde
{G}/M$ of $B(V)\#\Bbbk G.$
\end{proof}

\begin{corollary}
\label{uni}Let $B(V)\#\Bbbk G$ be the bosonization for link-indecomposable
$V\in%
\genfrac{}{}{0pt}{}{\Bbbk G}{\Bbbk G}%
\mathcal{YD}.$ Let $\tilde{G}$ be the universal coalgebra covering group of
$B(V)$ and write $G=\tilde{G}/N$ where $N\vartriangleleft\tilde{G}$. The
universal covering Hopf algebra of $B(V)\#\Bbbk G$ is $B(V)\#\Bbbk\tilde
{G}/[N,\tilde{G}].$
\end{corollary}

\section{Racks and Nichols algebras}

Let $B(V)$ be a Nichols algebra generated by the Yetter-Drinfeld module $V\in%
\genfrac{}{}{0pt}{}{\Bbbk G}{\Bbbk G}%
\mathcal{YD}$. We utilize the description of $B(V)$ as a subcoalgebra of the
cotensor coalgebra $C(V)$ of the bicomodule $V$. Namely the graded component
$B(V)(n)$ is the image of the quantum symmetrizer $\mathfrak{S}_{n}%
=\sum_{\sigma\in\mathbb{S}_{n}}\hat{\sigma}$ where $\symbol{94}$ denotes the
Matsumoto section $\mathbb{S}_{n}\rightarrow\mathbb{B}_{n}$. In particular,
the quadratic summand $B(V)(2)\subset V\otimes V$ is the image of $1+c.$ For a
fixed basis of $V$, we obtain an embedding $B\hookrightarrow\Bbbk Q$ where
$Q=Q_{V}$ is the $\dim V$-loop quiver with arrows labelled by basis elements
of $V$ and single vertex $1_{B(V)}$. In fact, we choose a basis of
$G$-homogeneous elements. Fixing this basis of homogeneous elements of $V$, we
thus identify $C(V)$ with the path coalgebra $\Bbbk Q.$ When $\dim\rho_{i}=1$
for all $i,$ the basis can be chosen to correspond to a union of conjugacy
classes of $G.$

As in \cite{Rosso98},\cite{Schau96}; cf. \cite{AS2002} $B(V)$ can be
constructed as both the subcoalgebra $\bigoplus\limits_{n\geq0}%
\operatorname{Im}\mathfrak{S}_{n}$ of the cotensor coalgebra $\Bbbk
Q=C(V)=k1\oplus V\oplus V\otimes V\oplus\cdots,$ and as a homomorphic image of
the tensor coalgebra $T(V)=\Bbbk^{Q}$ modulo the ideal $\bigoplus
\limits_{n\geq0}\ker\mathfrak{S}_{n}.$

\subsection{Coverings of Nichols algebras}

We assume that $V=(V,c=c^{q})=\Bbbk X$ is a Yetter-Drinfeld module with finite
rack $X=(X,\vartriangleright),$ with rack structure map $\vartriangleright
:X\times X\rightarrow X$ and $2$-cocycle $q:X\otimes X\rightarrow\Bbbk
^{\times}$ is as in \cite{AG2003a}, where\textit{\ we insist on a
one-dimensional image for }$q$. By \cite[Theorem 4.14]{AG2003a} such a braided
vector space arises as a Yetter-Drinfeld module a one-dimensional module
$\rho$ over the centralizer $G_{g}$ of an fixed chosen element $g\in G$. Fix a
basis $\{v_{x}|x\in X\}$ for $V$ where $v_{x}\in V_{x}$ for all $x.$ Note that
the assumption entails that the subpaces $V_{x}=\Bbbk v_{x},$ $x\in X$ are
one-dimensional. The braiding $c=c^{q}:V\otimes V\rightarrow V\otimes V$ is
defined by $c(v_{x}\otimes v_{y})=q(x,y)v_{y}\otimes v_{x}.$ We shall use the
same symbol for the map $c:X\times X\rightarrow X\times X,$
$c(x,y)=(x\vartriangleright y,x),$ as in \cite{GHV2011}. The Yetter-Drinfeld
module is thus $\oplus_{i}M(g_{i}$,$\rho_{i})=\oplus_{i}\Bbbk G\otimes_{\Bbbk
G_{g_{i}}}\rho_{i}.$ The group $G$ can be chosen to be finite if the subgroup
of $\Bbbk^{\times}$ generated by the $q(x,y)$ is finite and $X$ is finite, cf.
\cite[Theorem 4.14]{AG2003a}. We shall investigate groups that give rise to
braided vector spaces $(V,c^{q})$ and the Nichols algebra $B(V)$.

We follow the set-up as in \cite{GHV2011}. We have the braided vector space
$V=(\Bbbk X,c),$ and we let $\mathcal{O}$ $(=\mathcal{O}_{x,y})$ denote
$c$-orbit (of $(x,y)$) in $X\times X$ . Set ~$V_{\mathcal{O}}=\sum
_{(s,t)\in\mathcal{O}}V_{s}\otimes V_{t}$ and note that $\theta_{i}%
:=c^{i}(v_{x}\otimes v_{y}),$ $i=0,1,..m-1$ is a basis for $V_{\mathcal{O}}.$

The \textit{enveloping group} $G_{X}$ of a rack $X$ is defined to be the
quotient of the free group on generators $\{g_{x}|x\in X\}$ by the relations
\[
g_{x}g_{y}=g_{x\vartriangleright y}g_{x}%
\]
$x,y\in X.$

Let $G$ be a group with link-indecomposable $V\in%
\genfrac{}{}{0pt}{}{\Bbbk G}{\Bbbk G}%
\mathcal{YD}.$ Then $V=(\Bbbk X,c)$ for a rack $X\subset G$ where $X$ is a
union of conjugacy classes of $G$ and Supp$_{G}V=$ $X$ generates $G$. Since
the defining relations of $G_{X}$ hold in $G,$ there is a surjective group
homomorphism $G_{X}\rightarrow G.$

\begin{theorem}
\label{full}Let $B(V)$ be a Nichols algebra with a full set of quadratic
relations. Then \newline(a) $G_{X}$ is the coalgebra covering group of
$B(V)$\newline(b) $B(V)\#\Bbbk G_{X}$ is the universal covering coalgebra of
$B(V)\subset\Bbbk Q_{V}$ and \newline(c) $B(V)\#\Bbbk G_{X}\rightarrow
B(V)\#\Bbbk G$ is the universal Hopf covering of $B(V)\subset\Bbbk Q_{V}.$
\end{theorem}

\begin{proof}
As mentioned above, $B(V)$ has quadratic component $B(V)(2)=\operatorname{Im}%
(1+c)\subseteq V\otimes V.$ For $v_{x}\in V_{x},v_{y}\in V_{y}$ with $x\neq
y\in X$ we have $c(x\otimes y)=q(x,y)(x\triangleright y)\otimes x.$ We need to
see that
\[
(1+c)(v_{x}\otimes v_{y})=v_{x}\otimes v_{y}+q(x,y)v_{x\triangleright
y}\otimes v_{x}%
\]
is a minimal element of Im$(1+c)$. For then we obtain the relation $\left[
v_{x}\right]  \left[  v_{y}\right]  =\left[  v_{x\triangleright y}\right]
\left[  v_{x}\right]  $ in $\pi_{1}(Q,1_{B(V)})$ per the definition of
$\tilde{Q}_{B(V)}=\frac{\pi_{1}(B(V),1)}{N(B(V),1)}.$ By \cite[Lemma
3.4]{AFGV2011} $V\in%
\genfrac{}{}{0pt}{}{\Bbbk G_{X}}{\Bbbk G_{X}}%
\mathcal{YD}$ so there is a group surjection $\tilde{G}\rightarrow G_{X}$ with
$\left[  v_{x}\right]  \longmapsto g_{x}.$ It follows that $\tilde{G}\simeq
G_{X}$

We claim that $1+c:V_{\mathcal{O}}\rightarrow V_{\mathcal{O}}$ is onto (and
thus bijective) if and only if $c^{m}(v_{i}\otimes v_{j})\neq(-1)^{m}%
(v_{i}\otimes v_{j}).$ Let $m$ be the order of $c$ in $\mathcal{O\subset
}X\mathcal{\times}X.$ The restriction of $1+c$ to $V_{\mathcal{O}}$ is given
by the $m\times m$ matrix
\[%
\begin{bmatrix}
1 & 0 & \cdots & 0 & \lambda\\
1 & 1 & 0 & \cdots & 0\\
0 & 1 & \ddots & 0{}\hspace{2pt}\cdots & 0\\
0 & \ddots & \ddots & 1 & 0\\
0 & 0 & 0 & 1 & 1
\end{bmatrix}
\]
$\allowbreak$with respect to the basis $\{\theta_{i}\}$, where $\lambda
\in\Bbbk^{\times}$ is such that $c^{m}(v_{i}\otimes v_{j})=\lambda
v_{i}\otimes v_{j}.$ One can see that the determinant is $1+\left(  -1\right)
^{m-1}$ $\lambda$, so $1+c$ is bijective if and only if $\lambda\neq\left(
-1\right)  ^{m-1}.$ When $1+c$ is not a bijection, then it is easy to see that
the $m-1$ elements
\[
\theta_{0}+\theta_{1},\theta_{1}+\theta_{2},\cdots,\theta_{m-2}+\theta_{m-1}%
\]
form a basis for the image of $1+c$ restricted to $V_{\mathcal{O}}.$ It
follows that the $\theta_{i}$ are minimal elements.
\end{proof}

\begin{remark}
When $\lambda=\left(  -1\right)  ^{m-1}$, $\ker(1+c)=\Bbbk\sum_{n=0}%
^{m-1}\left(  -1\right)  ^{n}\theta_{n}\ $as noted in $\cite{GHV2011}.$ When
$\lambda\neq\left(  -1\right)  ^{m-1}$ then $\ker(1+c)|_{\mathcal{O}}=0.$
\end{remark}

\begin{remark}
The universal covering Hopf algebra was seen to be $B(V)\#\tilde{G}/[\tilde
{G},N]$ in Theorem \ref{uni}. Therefore $\tilde{G}/[\tilde{G},N]\simeq G_{X}$
for all choices of $N$ and it follows that $[\tilde{G},N]=\ker(\tilde
{G}\rightarrow G_{X}).$ In case $N$ is central, we get the result
$G_{X}=\tilde{G}$ as in the conclusion of the Theorem. But if there is not a
full set of quadratic relations, we may still get $G_{X}=\tilde{G},$ as can be
shown for Nichols algebras of finite Cartan type, e.g. last example in
$\ref{nonabelian ex}.$
\end{remark}

\section{Examples}

\subsection{rank 1}

Let $G=C_{m}=<K>$ be the cyclic group of order $m$ and let $q$ be an $m^{th}$
root of $1$. Let $H$ be the Hopf algebra with generators $E,$ $K$, where $K$
is group-like and $E$ is $(K,1)$ skew-primitive, and with relations $KE=qEK,$
$E^{m}=0.$ Then $H$ is the bosonization $B(\Bbbk E)\#\Bbbk C_{m}$ where
$B(\Bbbk E)=\Bbbk\lbrack E]/(E^{m})$ with $G$-comodule structure given by
$\delta(E)=K\otimes E$ and $G$-module structure given by $KE=qE$. Then $\Bbbk
E\in%
\genfrac{}{}{0pt}{}{\Bbbk G}{\Bbbk G}%
\mathcal{YD}$ and $B(\Bbbk E)\#C_{m}$ is a Hopf algebra. Its quiver is a
directed cycle $Q_{m}$ of length $m.$ The finite Hopf coverings of $B(\Bbbk
E)\#\Bbbk C_{m}$ of have underlying coalgebra coverings $B(\Bbbk
E)\#C_{n}\rightarrow$ $B(\Bbbk E)\#C_{m}$ with $m|n.$ Setting $n=\infty$ (so
$G=<g>$ is infinite cyclic) results in the universal covering Hopf algebras
whose quiver is of type $\mathbb{A}_{\infty}^{\infty}$ with unidirectional
arrows. These Hopf algebras appear in \cite{DIN2013}.

One may also consider the path coalgebras of the cyclic quivers $Q_{n},$ and
quivers of type $\mathbb{A}_{\infty}^{\infty}$ with unidirectional arrows for
$m\in\mathbb{Z}^{+}\cup\{\infty\}.$ As in \cite{CiRo2002} the path coalgebras
can be furnished with a Hopf algebra structure depending on an $m$th root of
unity $q$ where where $m|n$ or $n=\infty$. The Hopf coverings for fixed $m$
(and an $m$th root of 1) are $\Bbbk Q_{n}\rightarrow\Bbbk Q_{m}$ where $m|n$
or $n=\infty.$ The Hopf algebras in the previous paragraph are the
bosonizations of Nichols algebras for these infinite dimensional Hopf algebras.

\subsection{Fomin-Kirillov algebras \label{FK Ex}}

Let $X$ be the rack of transpositions of $\mathbb{S}_{n},$ $n>2$ and consider
the versions of $V=(\Bbbk X,c^{q})$ for the cocycles as in \cite{GHV2011}.
Then the quadratic version $\hat{B}_{2}(V)$ is the algebra from \cite{FK99}%
,\cite{MiS2000}; cf. \cite{AS2002}. It is known that $B(V)=\hat{B}_{2}(V)$ is
finite dimensional for $n\leq5,$ but this problem has been open for $n>5$ for
more than a decade.

One can enumerate the orbits (including orbits of size 1) of $c$ on $X\times
X$ as follows:%

\[%
\begin{tabular}
[c]{|c|c|}\hline
Orbits of size & \#orbits\\\hline
1 & $\binom{n}{2}=d$\\\hline
2 & $\frac{1}{2}\binom{n}{2}\binom{n-2}{2}$\\\hline
3 & $2\binom{n}{3}$\\\hline
\end{tabular}
\]

The total number of $c$-orbits is
\begin{align*}
f(n)  &  =\frac{n\left(  n-1\right)  }{2}+\frac{n\left(  n-1\right)  \left(
n-2\right)  \left(  n-3\right)  }{8}+\frac{n\left(  n-1\right)  \left(
n-2\right)  }{3}\\
&  =\allowbreak\frac{1}{24}n\left(  3n^{3}-10n^{2}+21n-14\right)  .
\end{align*}

The number of orbits in excess of $\binom{d}{2}$ is%
\begin{align*}
f(n)-\binom{d}{2}  &  =\frac{1}{24}n\left(  3n^{3}-10n^{2}+21n-14\right) \\
&  -\allowbreak\frac{1}{8}n\left(  n-1\right)  \left(  n-2\right)  \left(
n+1\right) \\
&  =\allowbreak-\frac{1}{6}n\left(  n-1\right)  \left(  n-5\right)
\end{align*}

There are fewer quadratic relations (=\#orbits) than $\binom{d}{2}$ (recall
$d=\binom{n}{2})$ for $n>5.$ There is a full set of quadratic relations, but
not "many" in the sense of \cite{GHV2011}.%

\[%
\begin{tabular}
[c]{|c|c|c|}\hline
$n$ & \#orbits in $X\times X=f(n)$ & \#orbits$-\binom{d}{2}$\\\hline
$3$ & $5$ & $2$\\\hline
$4$ & $17$ & $2$\\\hline
$5$ & $45$ & $0$\\\hline
$6$ & $100$ & $-5$\\\hline
\end{tabular}
\]

\subsection{Nonabelian group type\label{nonabelian ex}}

We adopt the list of some racks from e.g. \cite{GHV2011},\cite{GVZoo}%
,\cite{MiS2000},\cite{HLV2012}. We will consider the racks $\mathcal{A}%
$,$\mathcal{B}$,$\mathcal{C}$,$\mathcal{T}$ and the affine racks
$\mathrm{Aff}(5,2),\mathrm{Aff}(5,3),\mathrm{Aff}(7,5),\mathrm{Aff}(7,3)$ with
$2$-cocyles as in the references which result in finite-dimensional Nichols algebras.

Let $\mathcal{S}_{n}$ denote the rack of transpositions in $\mathbb{S}_{n},$
$n\geq3$ (as above). Let $\mathcal{B}$ be the rack of $4$-cycles in
$\mathbb{S}_{4}.$ Let $\mathcal{D}_{4}$ denote the rack of $4$ reflections
(transpositions) in the dihedral group $\mathbb{D}_{4}$ (order $8$). Hence
$\mathcal{D}_{3}=\mathcal{S}_{3}$, $\mathcal{A=S}_{4}$ and $\mathcal{C}%
=\mathcal{S}_{5}.$ Let $d=\dim V$. All but the last two rows have
indecomposable racks corresponding to irreducible Yetter-Drinfeld modules. All
but the last row have Nichols algebras with a full set of quadratic relations.
The last two rows have decomposable racks. The example over $\mathbb{D}_{4}$
is from \cite[Example 6.5]{MiS2000}; since the center of $D_{4}$ acts
trivially on the Yetter -Drinfeld module $V,$ the braiding reduces to the
Klein $4$-group $\mathbb{V}$ and the $\mathbb{D}_{4}$ bosonization is a double
cover of the smaller Hopf algebra over $\mathbb{V}$.

The two newer examples in \cite[ Prop. 32, 36]{HLV2012} (over $\mathcal{D}%
_{3}$ and $\mathcal{T}$ in rows 2 and 5 in table below, respectively) do have
a full set of quadratic relations, but do not have relations of form $x^{2}$.

In the last row, the Nichols algebra of type finite Cartan type of rank $2$ is
seen to have no quadratic relations (where the root of unity has order
$>$%
$2$) because the Serre relations are cubic. It can be shown that $G_{X}%
=\tilde{G}$ and is free abelian for Nichols algebras of finite Cartan type.

The computation of the enveloping groups and there centers was done with the
aid of GAP, or done by hand. The fact concerning $G_{\mathbb{S}_{n}}$and its
center are from \cite[Prop. 3.2]{AFGV2011}.%

\[%
\begin{tabular}
[c]{|r|r|r|r|r|r|}\hline
Rack $X$ & rank $d$ & $Z(G_{X})$ & $G_{X}/Z(G_{X})$ & $\#$orbits &
\#QR\\\hline
$\mathcal{S}_{n}$ & $\binom{n}{2}$ & $C_{\infty}$ & $\mathbb{S}_{n}$ & $f(n) $
& $f(n)$\\\hline
$\mathcal{D}_{3}$ & $3$ & $C_{\infty}$ & $\mathbb{S}_{n}$ & $5$ & $2$\\\hline
$\mathcal{B}$ & $6$ & $C_{\infty}$ & $\mathbb{S}_{4}$ & $17$ & $17$\\\hline
$\mathcal{T}$ & $4$ & $C_{\infty}\times C_{2}$ & $\mathbb{A}_{4}$ & $8$ & $8
$\\\hline
$\mathcal{T}$ & $4$ & $C_{\infty}\times C_{2}$ & $\mathbb{A}_{4}$ & $8$ & $4
$\\\hline
$\mathrm{Aff}(5,2)$ & $5$ & $C_{\infty}$ & $C_{5}\rtimes C_{4}$ & $10$ & $10
$\\\hline
$\mathrm{Aff}(5,3)$ & $5$ & $C_{\infty}$ & $C_{5}\rtimes C_{4}$ & $10$ & $10
$\\\hline
$\mathrm{Aff}(7,3)$ & $7$ & $C_{\infty}$ & $C_{7}\rtimes C_{6}$ & $21$ & $21
$\\\hline
$\mathrm{Aff}(7,5)$ & $7$ & $C_{\infty}$ & $C_{7}\rtimes C_{6}$ & $21$ & $21
$\\\hline
$\mathcal{D}_{4}$ & $4$ & $C_{\infty}\times C_{\infty}\times C_{2}$ &
$C_{2}\times C_{2}$ & $4$ & $4$\\\hline
rank 2 & $2$ & $C_{\infty}\times C_{\infty}$ &  & $2$ & $0$\\\hline
\end{tabular}
\]

\bibliographystyle{abbrv}
\bibliography{chinsrefs2014}

\end{document}